%% file: gbz-arxiv.tex
\documentclass[a4paper, 11pt]{article}
\usepackage{amsmath}
\usepackage{amsthm}
\usepackage{amssymb}
\usepackage{amsfonts}
\usepackage{authblk}
\usepackage[normalem]{ulem}
\usepackage{algorithm}
\usepackage[]{algpseudocode}
\usepackage{enumitem}
\usepackage{graphicx}
\usepackage{colortbl}
\usepackage{tikz}
\usetikzlibrary{arrows,matrix,fit,positioning}

\usepackage{listings}
\newtheorem{theorem}{Theorem}
\newtheorem{lemma}[theorem]{Lemma}

\newtheorem{proposition}[theorem]{Proposition}
\newtheorem{definition}[theorem]{Definition}

\newtheorem{example}[theorem]{Example}
\newtheorem{remark}[theorem]{Remark}

\newtheorem{convention}[theorem]{Convention}

\setlist{nosep} 

\providecommand{\keywords}[1]{\textbf{\textit{Keywords---}} #1}

\usepackage{xspace}
\usepackage{numprint}

\usepackage[colorinlistoftodos]{todonotes}
\usepackage{marginnote}

\newcommand{\bba}{\textsf{BBA}\xspace}
\newcommand{\nfn}{\textsf{NF}\xspace}
\newcommand{\magma}{\textsc{Magma}\xspace}
\newcommand{\macaulay}{\textsc{Macaulay2}\xspace}
\newcommand{\singular}{\textsc{Singular}\xspace}
\newcommand{\gbl}{\textsc{GB}\xspace}

\usepackage{url,amsmath}
\usepackage{booktabs}
\usepackage{multirow}

\newcommand*{\defeq}{\mathrel{\vcenter{\baselineskip0.5ex \lineskiplimit0pt
                     \hbox{\scriptsize.}\hbox{\scriptsize.}}}%
                     =}

\DeclareMathOperator{\headSym}{lt}
\DeclareMathOperator{\headmSym}{lm}

\DeclareMathOperator{\lcm}{lcm}
\newcommand{\lc}[1]{\mathrm{lc}\left({#1}\right)}
\newcommand{\hd}[1]{\headSym\left({#1}\right)}
\newcommand{\hm}[1]{\headmSym\left({#1}\right)}

\newcommand{\ecart}[1]{\ensuremath{\mathrm{ecart}\left({#1}\right)}}

\newcommand{\spoly}[2]{\ensuremath{\mathrm{spoly}\left({#1},{#2}\right)}}
\newcommand{\gpoly}[2]{\ensuremath{\mathrm{gpoly}\left({#1},{#2}\right)}}
\newcommand{\nf}[2]{\ensuremath{\text{\sffamily{NF}}\left({#1},{#2}\right)}}

\newcommand{\er}{\ensuremath{\mathcal{R}}\xspace}
\newcommand{\epr}{\ensuremath{\mathcal{P}}\xspace}
\newcommand{\numberGenerators}{m}
\newcommand{\ip}[1]{\left({#1}_1,\ldots,{#1}_\numberGenerators\right)}

\newcommand{\module}{\ensuremath{\epr^\numberGenerators}\xspace}

\newcommand{\const}{a}

\newcommand{\proj}[1]{\overline{#1}}
\newcommand{\mbasis}[1]{\boldsymbol{e}_{#1}}

\newcommand{\syz}[1]{\mathrm{syz}\left({#1}\right)}

\newcommand{\ltr}{\ensuremath{\headSym}-reduction\xspace}
\newcommand{\ltrs}{\ensuremath{\headSym}-reductions\xspace}

\newcommand{\lcr}{\ensuremath{\mathrm{lc}}-reduction\xspace}
\newcommand{\lcrs}{\ensuremath{\mathrm{lc}}-reductions\xspace}

\newcommand{\spt}{S-poly\-no\-mial\xspace}
\newcommand{\spts}{S-poly\-no\-mials\xspace}
\newcommand{\gpt}{GCD-poly\-no\-mial\xspace}
\newcommand{\gpts}{GCD-poly\-no\-mials\xspace}

\newcommand{\stb}{standard basis\xspace}
\newcommand{\stbs}{standard bases\xspace}
\newcommand{\gb}{Gr\"obner basis\xspace}
\newcommand{\gbs}{Gr\"obner bases\xspace}
\newcommand{\rpc}{{\sffamily RPC}\xspace}
\newcommand{\sbba}{{\sffamily sBBA}\xspace}

\newcommand{\F}{\mathbb{F}}

\newcommand{\Q}{\mathbb{Q}}
\newcommand{\Z}{\mathbb{Z}}
\newcommand{\N}{\mathbb{N}}

\definecolor{mygreend}{HTML}{2ca92c}
\definecolor{myredd}{HTML}{c31313}

\usepackage{caption}
\definecolor{codegreen}{rgb}{0,0.6,0}
\definecolor{codegray}{rgb}{0.5,0.5,0.5}
\definecolor{codepurple}{rgb}{0.58,0,0.82}
\definecolor{backcolour}{rgb}{0.99,0.99,0.99}

\renewcommand*{\arraystretch}{1.5}

\begin{document}


\title{Standard Bases over Euclidean Domains}

\author{Christian Eder\thanks{ederc@mathematik.uni-kl.de}}
\author{Gerhard Pfister\thanks{pfister@mathematik.uni-kl.de}}
\author{Adrian Popescu\thanks{popescu@mathematik.uni-kl.de}}
\affil{TU Kaiserslautern\\Department of Mathematics\\D-67663 Kaiserslautern}


%

 \maketitle

\begin{abstract}
In this paper we state and explain techniques useful for the computation of
strong Gr\"obner and standard bases over Euclidean domains: First we investigate
several strategies for creating the pair set using an idea by Lichtblau.
Then we explain methods for avoiding coefficient growth using syzygies.
We give an in-depth discussion on normal form computation resp. a generalized reduction
process with many optimizations to further avoid large coefficients. These are
combined with methods to reach \gpts at an earlier stage of the computation.
Based on various examples we show that our new implementation in the
computer algebra system \singular is, in general, more efficient than other
known implementations.
\end{abstract}


\input{intro}
\input{notation}
\input{pairs}
\input{coefficients}
\input{nf}

 \input{results}
 \input{conclusion}
\bibliographystyle{abbrv}

\end{document}

%% file: intro.tex
\section{Introduction}
\label{sec:intro} In $1964$ Hironaka already investigated
computational approaches towards singularities and introduced the notion of
\stbs for local monomial orders,\footnote{See Definition~\ref{def:orders}.} see, for
example,~\cite{hironaka11964, hironaka21964, grauert1972}.
In~\cite{bGroebner1965, bGroebner1965eng}, Buchberger initiated, in $1965$, the theory
of \gbs for global monomial orders by which many fundamental problems in mathematics, science
and engineering can be solved algorithmically. Specifically, he introduced some key
structural theory, and based on
this theory, proposed the first algorithm for computing \gbs.
Buchberger's algorithm introduced the concept of critical pairs and repeatedly carries out a certain
polynomial operation (called reduction).

Many of those reductions would be
determined as ``useless'' (i.e. no contribution to the output of the algorithm),
but only a posteriori, that is, after an (often expensive) reduction process.
Thus intensive research was carried out, starting with Buchberger, to avoid the
useless reductions via a priori criteria, see,  for
example,~\cite{bGroebnerCriterion1979,buchberger2ndCriterion1985,gmInstallation1988}.

Once the underlying structure is no longer a field, one needs
the notion of strong \gbs resp. strong \stbs. Influential work was done
by~\cite{kapur1988}, introducing the first generalization of Buchberger's
algorithm over Euclidean domains computing strong \gbs. Since then only a few
optimizations has been introduced, see, for example,~\cite{Wienand2011,
    lichtblau2012, eppSigZ2017}.

In this paper we introduce several new optimizations to the computation of
strong \stbs over Euclidean domains. In Section~\ref{sec:notation} we give the
basic notation and introduce the idea of a reduction step, generalized from the
field case. We state
Buchberger's algorithm over Euclidean domains for global and also for local
monomial orders. Section~\ref{sec:pairs} discusses different variants of how to
handle \spts and \gpts, especially generalized variants of Buchberger's product
and chain criterion. Over Euclidean domains like the integers, coefficient swell
and the missing normalization of the lead coefficient play an important role when it
comes to practical and efficient computation. Modular computation are not
possible in general, but we give a new attempt for keeping coefficients small
in Section~\ref{sec:coefficients}. In Section~\ref{sec:nf} we finally give an
in-depth discussion on the normal form computation which provides various attempts
like lead term reductions and lead coefficient reductions. This, again, helps to keep coefficients small and
minimizes the number of polynomials in a basis which have the same
leading monomial by applying efficient gcd computation.
We have done a new implementation of Buchberger's algorithm in the computer
algebra system \singular. In Section~\ref{sec:results} we compare our implementation
with \macaulay and \magma, exploring the impact of the above ideas with
some interesting results.

%% file: notation.tex
\section{Basic notations}
\label{sec:notation}
Let \er be a Euclidean domain without zero divisors.\footnote{The reader can
    feel fre to think of $\er = \Z$.}
A \emph{polynomial}
in $n$ variables $x_1,\ldots,x_n$ over \er is a finite \er-linear
combination of \emph{terms} $\const_{v_1,\ldots,v_n} \prod_{i=1}^n
x_i^{v_i}$,
$f=\sum_{v}\const_v x^v \defeq
\sum_{v\in\N^n}^{\text{finite}}\const_{v_1,\ldots,v_n} \prod_{i=1}^n
x_i^{v_i}$,
such that $v \in \N^n$ and $\const_v \in \er$.
The \emph{polynomial ring} $\epr \defeq \er[x] \defeq \er[x_1,\ldots,x_n]$
in $n$ variables over $\er$ is the set of all polynomials over $\er$
together with the usual addition and multiplication. For $f=\sum_{v}\const_v
x^v \neq 0 \in\epr$ we define \emph{the degree of $f$} by $\deg(f) :=
\max\left\{v_1+\cdots +v_n \mid \const_v \neq 0\right\}$. For $f=0$ we set
$\deg(f):=
-1$.
%

Let $\ip f \in \epr$ be a finite sequence of polynomials. We define a module
homomorphism $\pi : \module \rightarrow \epr$ by $\mbasis i \mapsto f_i$ for
all $1\leq i \leq \numberGenerators$. We use the shorthand notation $\proj\alpha
\defeq \pi(\alpha) \in \epr$ for $\alpha \in \module$. An element $\alpha \in
\module$ with $\proj\alpha = 0$ is called a \emph{syzygy}. The module of all
syzygies (of $\langle f_1,\ldots, f_m\rangle$) is denoted $\syz{\langle
  f_1,\ldots, f_m\rangle}$.

In the following we discuss computation with respect to different monomial orders:
\begin{definition}
\label{def:orders}
Let $<$ denote a monomial order on \epr. 
\begin{enumerate}
\item $<$ is called \emph{global} if $x^\alpha \geq 1$ for all $\alpha \in \N^n$.
\item $<$ is called \emph{local} if $x^\alpha \leq 1$ for all $\alpha \in \N^n$.
\item Moreover, we call $<$ \emph{mixed} if there exist $\alpha, \beta \in \N^n$ such
that $x^\alpha \leq 1 \leq x^\beta$.
\end{enumerate}
\end{definition}
Given such a monomial order $<$ we can highlight the maximal terms of
elements in \epr w.r.t. $<$: For $f\in \epr\ \backslash\ \{0\}$,
$\hd f$ is the \emph{lead term}, $\hm f$ the \emph{lead monomial}, and $\lc f$
the \emph{lead coefficient} of $f$. For any set $F \subset \epr$ we define
the \emph{lead ideal} $L(F) = \langle \hd f \mid f \in F\rangle$; for an ideal
$I \subset \epr$,\  $L(I)$ is defined as the ideal of lead terms of all elements of
$I$. Moreover, we define the \emph{ecart} of $f$ by $\ecart f := \deg(f) -
\deg\left(\hm f\right)$. 

Working over a field there are many equivalent definitions of how to obtain a
canonical or normal form when reducing a given polynomial by a Gr\"obner basis $G$.
Working over more general rings these definitions are no longer equivalent and
over Euclidean domains, like the integers, this, in particular, results in the term
of \emph{strong}ness we give a meaning in the following:

Assuming that our coefficient ring \er is an Euclidean domain we can define a total
order $\prec$ using the Euclidean norm $|\cdot|$ of its elements: Let $\const_1, \const_2
\in \er$, then $\const_1 \prec \const_2 \text{ if } |\const_1 | < | \const_2 |.$
For example, for the integers we can use the absolute value and break ties via
sign:
\[0\prec -1 \prec 1 \prec -2 \prec 2 \prec -3 \prec 3 \prec \ldots \]

The reduction process of two polynomials $f$ and $g$ in \epr depends now on the
uniqueness of the minimal remainder in the division algorithm in \er:

\begin{definition}
\label{def:reduction}
Let $f, g \in \epr$ and let $G= \{g_1,\ldots,g_r\} \subset \epr$ be a finite set of
polynomials.
\begin{enumerate}
\item We say that \emph{$g$ top-reduces $f$} if $\hm g \mid \hm f$ and if there exist
$a,b \in \er$ such that $\lc f = a\, \lc g +b$ such
that $a \neq 0$, which coincides with $b \prec \lc f$. The top-reduction of $f$
by $g$ is then given by
\[f - a \frac{\hm f}{\hm g} g.\]
So a top-reduction takes place if the reduced polynomial will have either a smaller lead
mononmial or a smaller lead coefficient.
\item Relaxing the reduction of the lead term to any term of $f$, we say
that \emph{$g$ reduces $f$}. In general, we speak of the reduction of a
polynomial $f$ w.r.t. a finite set $F\subset \epr$.
Let 
\item We say that $f$ has a \emph{weak standard representation} w.r.t. $G$ if
$f = \sum_{i=1}^r h_i g_i$ for some $h_i \in \epr$ such that $\hm f = \hm{h_j
  g_j}$ for some $j \in \{1,\ldots,r\}$.
\item We say that $f$ has a \emph{strong standard representation} w.r.t. $G$ if
$f = \sum_{i=1}^r h_i g_i$ for some $h_i \in \epr$ such that $\hm f = \hm{h_j
  g_j}$ for some $j \in \{1,\ldots,r\}$ and $\hm f > \hm{h_k g_k}$ for all $k
  \neq j$.
\end{enumerate}
\end{definition}

This kind of reduction is equivalent to CP3 from~\cite{kapur1988} and generalizes
Buchberger's attempt from~\cite{buchberger2ndCriterion1985}.

The result of such a reduction might not be unique. This uniqueness is exactly
the property \emph{\stbs} give us. Before defining \stbs, let us give a short note on
the naming convention in this paper:

\begin{convention}
Note that the term \emph{\gb} was introduced by Buchberger in 1965 for bases
w.r.t. a global monomial order (\cite{bGroebner1965, bGroebner1965eng}). Independently, Hironaka
(\cite{hironaka11964, hironaka21964}), and, again
independetly, Grauert (\cite{grauert1972}), developed an analgous
concept, called \emph{\stb}, for multivariate power series, i.e. for polynomial
rings equipped with a local monomial order. For this paper we decided to use
the notion \emph{\stb} since it is nowadays the more general one.
\end{convention}

\begin{definition}
\label{def:strong-sb}
A finite set $G \subset \epr$ is called a \emph{\stb} for an ideal $I$ w.r.t.
a monomial order $<$ if $G \subset I$ and $L(G) = L(I)$. Furthermore, $G$ is called a
\emph{strong \stb}
   if for any $f \in I\backslash\{0\}$ there exists a $g\in G$
such that $\hd g \mid \hd f$. 
\end{definition}

\begin{remark}
Note that $G$ being a strong \stb is equivalent to all elements $g \in G$ having
a strong standard representation w.r.t. $G$. See, for
example, Theorem~1 in~\cite{lichtblau2012} for a proof.
\end{remark}

Clearly, assuming the field case, any \stb is a strong \stb. But in our setting
with \er being an Euclidean ring one has to check the coefficients, too, as
explained in Definition~\ref{def:reduction}.

\begin{example}
\label{ex:stronggb}
Let $\er = \Z$ and $I= \langle x \rangle \in \er[x]$. Clearly, $G :=
\{2x,3x\}$ is a \stb for $I$: $L(I) = \langle x \rangle$ and $x = 3x-2x \in
L(G)$. But $G$ is not a strong \stb for $I$ since $2x \nmid
x$ and $3x \nmid x$.
\end{example}

In order to compute strong \stbs we need to consider two different types of
special polynomials:

\begin{definition}
\label{def:spoly}
Let $f,g \in \epr$. We assume w.l.o.g. that $\lc f \prec \lc g$.
Let $t = \lcm\left(\hm f, \hm g\right)$, $t_f = \frac{t}{\hm f}$, and.$t_g = \frac{t}{\hm
  g}$.
\begin{enumerate}
\item Let $a = \lcm\left(\lc f, \lc g\right)$, $a_f = \frac{a}{\lc g},$ and $a_g =
\frac{a}{\lc f}$. The \emph{\spt} of $f$ and $g$ is denoted 
\[\spoly f g = a_f t_f f - a_g t_g g.\]
\item Let $b = \gcd\left(\lc f, \lc g\right)$ . Choose $b_f, b_g \in \er$ such that $b = b_f \lc f + b_g \lc
g$.\footnote{Since $\er = \Z$ is an Euclidean ring the extended gcd always exists.}
The \emph{\gpt} of $f$ and $g$ is denoted 
\[\gpoly f g = b_f t_f f + b_g t_g g.\]
\end{enumerate}
\end{definition}

\begin{remark} \
\label{rem:gpairs}
\begin{enumerate}
\item In the field case we do not need to consider \gpts at all since we can
always normalize the polynomials, i.e. ensure that $\lc f = 1$.
\item Note that $\gpoly f g$ is not uniquely defined: Working over $\er = \Z$ we
know that we can write $\langle \lc f, \lc g \rangle$ as a principal ideal, say
$\langle c \rangle = \langle \lc f, \lc g \rangle$ for some $c\in\er$. Then
there exist $c_f \neq c_f'$, $c_g \neq c_g' \in \er$ such that
\[c_f \lc f + c_g \lc g = c = c_f' \lc f + c_g' \lc g.\]
Depending on the implementation of the $\gcd$ algorithm one specific choice is
made for each \gpt.
\end{enumerate}
\end{remark}

From Example~\ref{ex:stronggb} it is clear that the usual Buchberger
algorithm as in the field case will not compute a strong \stb as we would only
consider $\spoly{2x}{3x} = 3 \cdot 2x - 2 \cdot 3x = 0$. Luckily, we
can fix this via taking care of the corresponding \gpt:
\[\gpoly{2x}{3x} = (-1) \cdot 2x - (-1) \cdot 3x = x.\]
It follows that given an ideal $I \subset \epr$ a strong \stb for $I$ can now be computed using
Buchberger's algorithm taking care not only of \spts but also of \gpts. We refer,
for example,  to~\cite{lichtblau2012} for more details.

In Algorithm~\ref{alg:bba} we give pseudo code for a generic Buchberger
algorithm over the integers. Here, no criterion for detecting useless elements
is applied. This is the topic of the next section. But what is necessary to
discuss beforehand is how to get strong standard representations of elements
handled by Algorithm~\ref{alg:bba}. This is the concept of a \emph{normal
form}:

\begin{definition}
Let $<$ be a monomial order on $\epr$.
Let $\mathcal G$ denote the set of all finite subsets $G \subset \epr$. We call
the map
\begin{center}
$
\begin{array}{rccc}
\text{\sffamily{NF}}: &\epr \times \mathcal G &\longrightarrow& \epr\\
                     & (f, G) &\longmapsto& \nf f G,
\end{array}
$
\end{center}
a \emph{weak normal form} w.r.t. $<$ if for all $f\in \epr$ and all $G \in
\mathcal G$ the following hold:
\begin{enumerate}
\item $\nf 0 G = 0$.
\item If $\nf f G \neq 0$ then $\hd{\nf f G} \notin L(G)$.
\item If $f \neq 0$ then there exists a unit $u\in \epr$ such that either
$uf = \nf f G$ or
$r = uf - \nf f G$ has a strong standard representation w.r.t. $G$.
\end{enumerate}
A weak normal form {\sffamily NF} is called a \emph{normal form} if we can always
choose $u=1$.
\end{definition}

Next we give algorithms that compute normal forms. For their correctness we
refer to Section~$1.6$ in~\cite{gpSingularBook2007}. Algorithm~\ref{alg:nfg}
presents a normal form algorithm for computation w.r.t. a global monomial order
$<$:

\begin{algorithm}
\caption{Normal form w.r.t. a global monomial order $<$
  (\nfn)} 
\label{alg:nfg}
\begin{algorithmic}[1]
\Require{Polynomial $f \in \epr$, finite subset $G\subset \epr$}
\Ensure{\nfn of $f$ w.r.t. $G$ and $<$}
\State{$h \gets f$}
\While{$\left(h \neq 0 \text{ and }G_h := \{g \in G \mid g \text{ top-reduces } h\}
  \neq\emptyset\right)$}
\State{Choose $g \in G_h$.}
\State{$h\gets $ Top-reduction of $h$ by $g$ (see
    Definition~\ref{def:reduction})}
\EndWhile
\State{$\text{\textbf{return }}h$}
\end{algorithmic}
\end{algorithm}

Note that Algorithm~\ref{alg:nfg} may enter an infinite \texttt{while} loop if applied
to local monomial orders. Let us illustrate this behaviour with the ``standard''
example. 

\begin{example}
Let $\epr = K[x]$ where $K$ is a field. We equip $\epr$ with a local monomial
order $<$, i.e. $x < 1$. We set $G = \{g\}$ where $g=x-x^2$ and we want to
compute $\nf f G$ where $f = x$. Using Algorithm~\ref{alg:nfg} we start setting
$h := f$ and find that $g$ top-reduces $h$:
\[h := x - \left(x-x^2\right) = x^2.\]
Now we can again top-reduce $h$ via subtracting $xg$:
\[h := x^2 - x\left(x-x^2\right) = x^3.\]
This process does not stop, but constructs a power series equation:
$x - \left(\sum_{i=0}^\infty x^i\right) \left(x-x^2\right) = 0.$
Since $x<1$ we know that $\sum_{i=0}^\infty x^i = \frac{1}{1-x}$. We see that
Algorithm~\ref{alg:nfg} computes correctly since $(1-x) x = x-x^2$,
still, it is not able to find the finite expression of the power series.
\end{example}

In~\cite{mora82} Mora gave the first attempt to achieve a terminating normal
form algorithm also for local monomial orders:

\begin{algorithm}[h]
\caption{Mora's normal form algorithm w.r.t. a local monomial order $<$
  (\nfn)} 
\label{alg:nfl}
\begin{algorithmic}[1]
\Require{Polynomial $f \in \epr$, finite subset $G\subset \epr$}
\Ensure{\nfn of $f$ w.r.t. $G$ and $<$}
\State{$h \gets f$}
\State{$T \gets G$}
\While{$\left(h \neq 0 \text{ and }T_h := \{g \in G \mid g \text{ top-reduces } h\}
  \neq\emptyset\right)$}
\State{Choose $g \in T_h$ with $\ecart g$ minimal.}
\If{$\left(\ecart g > \ecart h\right)$}
\State{$T \gets T \cup \{h\}$}
\EndIf
\State{$h\gets $ Top-reduction of $h$ by $g$ (see
    Definition~\ref{def:reduction})}
\EndWhile
\State{$\text{\textbf{return }}h$}
\end{algorithmic}
\end{algorithm}

\begin{example}
\label{ex:local-order}
Let $\epr = \Z[x,y]$. A strong \stb for the ideal $I=\left\langle 6+y+x^2,
4+x\right\rangle \subset \epr$ w.r.t. negative degree reverse
lexicographical order $<$ is given by
\[G = \left\{2-x+y+x^2, x-2y-x^2-xy-x^3\right\}.\]
Since $4+x \in I$ we assume that $\nf {4+x} G = 0$. In
Table~\ref{table:infinite-nf} we state Mora's normal form computation
with notation as in Algorithm~\ref{alg:nfl}, i.e. if we have a choice for the
reducer, we take the one of minimal possible ecart.
We start with $h=4+x$ and $T_h =
G$.
\begin{table}[h!]
	\centering
  \def\arraystretch{1.2}
    \begin{tabular}{c|c|c}
    \toprule
    \multicolumn{1}{c|}{$h$} &
    \multicolumn{1}{c|}{$g$} &
    \multicolumn{1}{c}{$h$ \emph{added to} $T_h$?}\\
    \midrule
    $4+x$ & $2-x+y+x^2$ & \checkmark\\
    $3x-2y-2x^2$ & $4+x$ & - \\
    $-x -2y - 3x^2$ & $x-2y-x^2-xy-x^3$ & \checkmark\\
    $-4y - 4x^2 - xy - x^3$ & $4+x$ & -\\
    $- 4x^2 - x^3$ & $4+x$ & -\\
    $0$ & - & -\\
    \bottomrule
    \end{tabular}
	\caption{Local normal form computation due to Mora}
	\label{table:infinite-nf}
\end{table}

Note the importance of $4+x$ being added to $T_h$. Without $4+x$ as reducer the
reduction process would not terminate.
\end{example}

\begin{remark}
Sometimes it can be more efficient to not only add the current status of $h$ to
$T_h$ as a new reducer, but also to generate \gpts with $h$. Doing so,
several lead coefficient reductions may be done in one step. See
Section~\ref{sec:nf} for more details.
\end{remark}
Now we are ready to state Buchberger's algorithm. For the theoretical background
of the algorithm (Buchberger's criterion) we refer to
Theorem~\ref{thm:buchberger-criterion} in Section~\ref{sec:pairs}.

\begin{algorithm}
\caption{Buchberger's algorithm for computing strong \stbs
  (\sbba)} 
\label{alg:bba}
\begin{algorithmic}[1]
\Require{Ideal $I=\langle f_1,\ldots,f_m\rangle \subset \epr$, monomial order
  $<$, normal form algorithm \nfn (depending on $<$)}
\Ensure{\gb $G$ for $I$ w.r.t. $<$}
\State{$G \gets \{f_1,\ldots,f_m\}$}
\State{$P \gets \left\{\spoly{f_i}{f_j}, \gpoly{f_i}{f_j} \mid 1 \leq i < j \leq
  m\right\}$}\label{alg:bba:update1}
\While{$\left(P \neq \emptyset\right)$}
\State{Choose $h \in P$, $P \gets P \setminus \{h\}$}\label{alg:bba:choose}
\State {$h \gets \nf h G$}
\If{$\left(h \neq 0\right)$}
\State{$P \gets \left\{\spoly{g}{h}, \gpoly{g}{h} \mid g \in
  G\right\}$}\label{alg:bba:update2}
\State{$G \gets G \cup \{h\}$}
\EndIf
\EndWhile
\State{$\text{\textbf{return }}G$}
\end{algorithmic}
\end{algorithm}

%% file: pairs.tex
\section{How to choose Pairs}
\label{sec:pairs}
Having two classes of polynomials to handle, namely \spts and \gpts we also need
criteria for deciding when such a polynomial is useless in the sense of
predicting a zero reduction or having already a strong standard representation.
The criteria presented in the following are then applied to
Algorithm~\ref{alg:bba} in lines~\ref{alg:bba:update1} and~\ref{alg:bba:update2}
when new elements for the pair set are generated.

A first criterion takes care of useless \gpts:

\begin{lemma}
Let $f,g \in \epr$ such that $\lc f \mid \lc g$. Then $\gpoly f g$ reduces
to zero w.r.t. $\{f,g\}$.
\label{crit:gpoly-easy}
\end{lemma}

\begin{proof}
Since $\gcd\left(\lc f, \lc g\right) = \lc f$ we can choose $b_f = 1$ and $b_g
=0$. Thus $\gpoly f g = 1 \cdot t_f \cdot f - 0 \cdot t_g \cdot g = t_f f$ for monomial
multiples $t_f, t_g$ such that $t_f \hm f = t_g \hm g = \lcm\left(\hm f, \hm
    g\right)$. It follows that $\gpoly f g$ just a multiple of $f$ and
thus reduces to zero w.r.t. $\{f,g\}$.
\end{proof}

As a next step we state well-known criteria by Buchberger, the Product and the
Chain criterion:

\begin{lemma}[Buchberger's Product Criterion]
Let\\
$f,g \in \epr$ be polynomials such that $\hm f$ and $\hm g$ are coprime and $\lc
f$ and $\lc g$ are coprime.
Then $\spoly f g$ reduces to zero w.r.t. $\{f,g\}$
\label{thm:prod-crit}
\end{lemma}

\begin{proof}
The proof is the same as in the field case.
See Theorem~3 in~\cite{lichtblau2012} for more details.
\end{proof}

Note that Lemma~\ref{thm:prod-crit} does \emph{not} apply to \gpts. Take, for
example, $G=\{f,g\}$ with $f=3x$ and $g=2y$. Then $\hm f = x$ and $\hm g = y$
are coprime and $\spoly f g = g f - fg = 0$, but $\gpoly f g = y \cdot 3x - x
\cdot 2y = xy$. Clearly, we cannot reduce $xy$ any further w.r.t. $G$.

\begin{lemma}[Buchberger's Chain Criterion]
Let $G \subset \epr$ finite and let $f,g,h \in G$ be polynomials such that
\begin{enumerate}
\item $\hm f \mid \lcm\left(\hm g, \hm h\right)$ and
\item $\lc f \mid \lcm\left(\lc g, \lc h\right)$.
\end{enumerate}
If $\spoly f g$ and $\spoly f h$ have a strong standard representation w.r.t.
$G$ then also $\spoly g h$ has a strong standard representation w.r.t. $G$.
\label{thm:chain-crit}
\end{lemma}

\begin{proof}
The lemma is proven as in the field case. We want to have an \spt chain
\[\spoly g h = c_{f,g} m_{f,g} \spoly f g + c_{f,h} m_{f,h} \spoly f h\]
for some coefficients $c_{f,g}, c_{f,h}$ and  some monomials $m_{f,g}, m_{f,h}$.
Property $(1)$ ensures proper monomial multiples $m_{f,g}$ and $m_{f,h}$.
Working over the integers we have to take care of the coefficients, too. Thus
property $(2)$ is needed to ensure the existence of proper multiples $c_{f,g}$
and $c_{f,h}$. For more details see, for example, the proof of Theorem~4
in~\cite{lichtblau2012}: There the statement is proven for a strengthened
version of property $(2)$, namely $\lc f \mid \lc g \mid \lc h$ resp. $\lc g
\mid \lc f \mid \lc h$.
\end{proof}

Moreover, we can state a similar criterion for \gpts:

\begin{lemma}
Let $G \subset \epr$ finite and let $f,g,h \in G$ be polynomials such that
\begin{enumerate}
\item $\hm f \mid \lcm\left(\hm g, \hm h\right)$ and
\item $\lc f \mid \gcd\left(\lc g, \lc h\right)$.
\end{enumerate}
Then $\gpoly g h$ has a strong standard representation w.r.t. $G$.
\label{thm:chain-crit-gcd}
\end{lemma}

\begin{proof}
See Theorem~10.11 in~\cite{bwkGroebnerBases1993} and Theorem~5
in~\cite{lichtblau2012}.
\end{proof}

Besides these statements one can ``decide'' for two polynomials $f$ and
$g$ if we need to compute the corresponding \gpt or the corresponding \spt.
This goes back to~\cite{kapur1988} and can be also found as a variant of
Buchberger's criterion as Theorem~2 in~\cite{lichtblau2012}.

\begin{theorem}[Variant of Buchberger's Criterion, Theorem~2
in~\cite{lichtblau2012}]
Let $G \subset \epr$ be a finite set of polynomials, let $<$ be a monomial order
on $\epr$. The following are equivalent:
\begin{enumerate}
\item $G$ is a strong standard basis w.r.t. $<$.
\item\label{cond:all} For all $f,g \in G$ $\spoly f g$ and $\gpoly f g$ reduce to zero.
\item\label{cond:only} For all $f,g \in G$ the following hold:
\begin{enumerate}
\item If $\lc f \mid \lc g$ or $\lc g \mid \lc f$ then $\spoly f g$ reduces to zero.
\item If $\lc f \nmid \lc g$ and $\lc g \nmid \lc f$ then $\gpoly f g$ reduces to zero.
\end{enumerate}
\end{enumerate}
\label{thm:buchberger-criterion}
\end{theorem}

Clearly, when implementing the above criteria especially the choice between
Condition~\ref{cond:all} and Condition~\ref{cond:only} in
Theorem~\ref{thm:buchberger-criterion} has a huge influence on the computation:
\begin{enumerate}
\item Depending the choice of the next element in the pair set in
Algorithm~\ref{alg:bba} it is obvious that Condition~\ref{cond:only} lies an
emphasis on \gpts. For a pair of polynomials $f,g\in G$ the algorithm tries to
keep of the lead coefficient of the genrated polynomial as small as possible.
This process goes on until at some point eventually this smaller lead
coefficient divides $\lc f$ $\lc g$. Then the corresponding \spt is generated
which then removes the whole lead term.
\item If we use Condition~\ref{cond:all} then there might be a lead term
cancellation, i.e. \spt, being handled before the complete reduction process of
the lead coefficient, i.e. handling of \gpts, is finished.
\end{enumerate}

Of course, one can have an influence on the above situation depending on the
choice of the next element from the pair set $P$ in Line~\ref{alg:bba:choose} of
Algorithm~\ref{alg:bba}. Lichtblau notes in~\cite{lichtblau2012} that, until
now,
there is no real comparison between the two attempts due to missing
implementations.

One statement we can make is the following:

\begin{proposition}
Theorem~\ref{thm:buchberger-criterion} already includes
Lemma~\ref{thm:prod-crit}.
\label{prop:prod-crit-useless}
\end{proposition}

\begin{proof}
By Theorem~\ref{thm:buchberger-criterion} we do only consider $\spoly f g$ if either
$\lc f \mid \lc g$ or $\lc g \mid \lc f$. Lemma~\ref{thm:prod-crit} applies only
in the other situations, but there no \spt is generated at all.
\end{proof}

We have implemented Buchberger's algorithm with both variants of Buchberger's
criterion in \singular. In Section~\ref{sec:results} we present more detailed
results.

%% file: coefficients.tex
\section{Avoiding Coefficient Swell}
\label{sec:coefficients}
One main problem when computing over the integers is coefficient growth. We
cannot normalize polynomials as usually done over fields. The only method to
keep at least lead coefficients as small as possible from inside the algorithm
is to efficiently compute \gpts as discussed in Section~\ref{sec:pairs}.

The first idea to handle coefficient swell might be to use modular methods as
done over fiels, see, for example,~\cite{arnoldModular2003}. Sadly this concept
is not working over the integers:

\begin{example}
Let $I = \langle 6x, 8x \rangle \subset \Z[x]$. A strong \stb w.r.t. a global
monomial order $<$ is $G = \{2x,6x,8x\}$ where $2x = \gpoly{8x}{6x}$. Next we try to
compute modular standard bases for $\langle 6x, 8x \rangle \subset \F_p[x]$
for some prime number $p > 3$.\footnote{We choose $p>3$ since it should at least
not divide the lead coefficients of the input polynomials.} The corresponding
\stb is $G_p = \{6x,8x\}$: Buchberger's algorithm over $\F_p$ only considers
$\spoly{8x}{6x} = 0$ and terminates afterwards, thus only the initial generators
are added to $G_p$. If
we ensure that Buchberger's algorithm computes a reduced\footnote{A
reduced \stb over a field w.r.t. a global monomial order means that the lead
coefficients of all elements in the basis are $1$ and no lead term divides
any term of any other polynomial in the basis.} \stb we then get $G_p =
\{x\}$. The problem is that in no case we would get $2x \in G_p$ which is
the important element for the strongness of the \stb for $I$ over $\Z$. Even
before applying Hensel lifting or the Chinese remainder theorem the information
needed is lost.

Thus there is no way to compute a strong \stb over $\Z$ via modular computations
over $\F_p$ and lifting techniques in general.
\end{example}

One trick we can do is trying to find monomials or constants in the ideal we
want to compute a strong \stb for. If we can add such elements to the list of
input polynomials of Algorithm~\ref{alg:bba} this can give a huge speed up to
the overall computation. The following lemma gives us a hint on how to do so.

\begin{lemma}
Let $<$ be a monomial order on $\Z[x]$ and let
$I=\langle f_1,\ldots, f_m\rangle \subset \Z[x]$. Let $\tilde I = \langle
f_1,\ldots, f_m\rangle \subset \Q[x]$. If the standard basis $G = \langle 1
\rangle $ for $\tilde I$ w.r.t. $<$ then there exists a constant $c \in \Z$
such that $c \in I$.
\label{lem:syz-coeff}
\end{lemma}

\begin{proof}
Let $G = \langle 1 \rangle \subset \Q[X]$. Consider $J = \langle 1, f_1,\ldots,
f_m\rangle \subset \Q[x]$. Consider the free module $\Q[x]^{m+1}$ with
standard generators $e_0,\ldots,e_m$ together
with the map $\pi: \Q[x]^{m+1} \rightarrow \Q[x]$ defined via $e_0 \mapsto
1$ and $e_i \mapsto f_i$ for all $1 \leq i \leq m$.
Since $1 \in G$ there must exist a syzygy $\sigma \in \Q[x]^{m+1}$ of the
following structure
\[ \sigma = e_0 + \sum_{i=1}^m p_i e_i\]
where $p_i \in \Q[x]$ are polynomials for all $1\leq i \leq m$.
In other words, we can represent $1 = \pi(e_0)$ as a $\Q[x]$-linear
combination of the $f_i = \pi(e_i)$. Moreover, we define
\[c := \lcm\left(\text{all denominators of all
coefficients of all terms of all $p_i$}\right).\]
Thus we get another syzygy $c\sigma$ which corresponds to the equation
\[c = c \cdot 1 = c \cdot \pi(e_0) =\sum_{i=1}^m (c p_i)\cdot \pi(e_i) =
\sum_{i=1}^m (c p_i) \cdot f_i\]
where $c p_i \in \Z[x]$ for all $1 \leq i \leq m$.
By construction it follows that $c \in I$.
\end{proof}

\begin{algorithm}
\caption{RationlPreCheck (\rpc)} 
\label{alg:rpc}
\begin{algorithmic}[1]
\Require{Ideal $I=\langle f_1,\ldots,f_m\rangle \subset \epr$, monomial order
  $<$}
\Ensure{Ideal $J$ such that $J = I$}
\State{$G \gets$ \stb for $\langle f_1,\ldots, f_m\rangle$ in $\Q[x]$ w.r.t. $<$}
\State{$S \gets \syz{\langle 1, f_1,\ldots, f_m\rangle} \subset \sum_{i=0}^m\Q[x]
  e_i$ w.r.t. $<$}
\If{$1 \in G$}
\State{Search for $\sigma \in S$ with $0$th component of the form $c e_0$, $c
  \in \Q$.}
\State{Find multiple $\lambda \in \Z$ such that $\lambda \sigma \in \sum_{i=0}^m
  \Z[x]e_i$}
\State{$J \gets \langle \lambda c, f_1, \ldots, f_m\rangle$}
\EndIf
\State{$\text{\textbf{return }}J$}
\end{algorithmic}
\end{algorithm}

\begin{example}
We give two examples, a small one we do by hand and a bigger one which gives a
real benefit for the overall computational time:
\begin{enumerate}
\item Let $I \subset \Z[x,y]$ be given by $I=\langle x+4, xy+9, x-y+8\rangle$.
We want to compute a strong \stb for $I$ w.r.t. the degree reverse-lexicographical
order $<$. The \stb for $I$ over $\Q$ includes $1$, so we have a constant in the
\stb for $I$ over $\Z$. We compute $\syz{\langle 1, x+4, xy+9, x-y+8\rangle}
\subset \sum_{i=0}^3 \Q[x,y]$ and get the following three syzygies:
\begin{center}
$
\begin{array}{rcl}
\sigma_1 & = & 7 e_0 - (x+4) e_1 + e_2 + x e_3,\\
\sigma_2 & = & (y-4) e_0 - e_1 + x e_3,\\
\sigma_3 & = & (x+4) e_0 - e_1.
\end{array}
$
\end{center}
$\sigma_1$ is the relation from which we can extract the corresponding constant
for $I$:
\begin{center}
$
\begin{array}{rcl}
7 = 7 \pi(e_0) &=& (x+4) \pi(e_1) - \pi(e_2) - x \pi(e_3)\\
  &=& (x+4)(x+4) - (xy+9) - x(x-y+8)\\
  &=& x^2+8x+16 - xy -9 - x^2 + xy -8x.
\end{array}
$
\end{center}
Thus, we add $7$ to the initial generators of $I$ and run
Algorithm~\ref{alg:bba}. We receive a strong \stb $G = \{7, x+4, y-4\} \subset
\Z[x,y]$ for $I$.
\item A bigger example is given as \texttt{rationalPreCheckExample()} in our
publicly available benchmark library~(\cite{singular-benchmarks}): The ideal $I$
we are considering is generated by $70$ polynomials in $\Z[x,y,z]$. We want to
compute a strong \stb for $I$ w.r.t. the degree reverse-lexicographical order $<$.
In Table~\ref{table:syz-example} give some characteristics of the computation
of a \stb for $I$ on an Intel Core
i7-5557U CPU with 16 GB RAM. Note that the computation of
Algorithm~\ref{alg:rpc} over $\Q$ takes $<0.01$ seconds and needs $<0.3$ MB of
memory, so it is
negligible compared to the computational cost over $\Z$. The strong \stb for $I$ computed by
our implementation consists of only $9$ elements
\begin{align*}
G = \{&18, 6z-12, 2y-4, 2z2+4z+8,\\
     &yz+z+3, 3x2z-15x2, x2y+3x2z+x2,\\
     &x3+10z, x2z2-4x2z-11x2\}
\end{align*}
From Algorithm~\ref{alg:rpc} we do not directly get the constant $18 \in G$, but
we get a multiple of it: $6,133,248$. Adding this constant to the generators of
$I$ and applying Algorithm~\ref{alg:bba} represents the third column of
Table~\ref{table:syz-example}, whereas a direct application of
Algorithm~\ref{alg:bba} on $I$ is given in column two.
\end{enumerate}
\end{example}

\begin{table}[h]
	\centering
  \def\arraystretch{1.2}
    \begin{tabular}{c||c|c}
    \toprule
    \multicolumn{1}{c||}{\textbf{Characteristics $\setminus$ Algorithms}} &
    \multicolumn{1}{c|}{\sbba} &
    \multicolumn{1}{c}{\rpc + \sbba}\\
    \midrule
    \text{maximal degree} & 13 & 13\\
    \text{\# zero reductions} & 1,130 & 795\\
    \text{\# product / chain criteria} & 1,279 / 2,990 & 826 / 1,925\\
    \text{memory usage (in MB)} & 1.51 & 0.78\\
    \bottomrule
    \end{tabular}
	\caption{Characteristics of \stb computations of Example~\ref{sec:syz-example}}
	\label{table:syz-example}
\end{table}
\begin{remark} \
\begin{enumerate}
\item Clearly, one can generalize Algorithm~\ref{alg:rpc}: If we do not have $1 \in G$
over $\Q$ we might still get a short polynomial, even a monomial whose
corresponding $\Z[x]$ representation can then be recovered from the corresponding
syzygy module. Note that in this case the reconstruction is a bit harder and the
precheck might take longer, since we have a more complex \stb
computation over $\Q$ that does not terminate early.
\item If one has a computer with at least two cores available the usage of
\texttt{parallel.lib} resp. \texttt{task.lib}~(\cite{singular-parallel,singular-tasks})
in \singular might be worthwhile: One could start the
direct computation of \sbba over $\Z$ on one core plus \rpc over $\Q$ on the
other core. Using the \texttt{waitfirst} command one could always ensure that the
fatest possible running time is achieved.
\end{enumerate}
\end{remark}

%% file: nf.tex
\section{Normal Form Computations}
\label{sec:nf}
Let us recall Definition~\ref{def:reduction} and give the two different
types of a reduction step a name:

\begin{definition}
Let $f \in \epr$ and let $G \subset \epr$ be a finite subset.
\begin{enumerate}
\item If there exists $g\in G$ such that $\hm g \mid \hm f$ and $\lc g \mid \lc
f$ then $f - \frac{\lc f}{\lc g} \frac{\hm f}{\hm g} g$ is a \emph{top-\ltr} of
$f$ (w.r.t. $g$).
\item If there exists $g\in G$ such that $\hm g \mid \hm f$, $\lc g \prec \lc f$ and $\lc g \nmid \lc
f$ then $f - a \frac{\hm f}{\hm g} g$ with $\lc f = a\, \lc g + b$ where $a,b
\in \Z$, $b \neq 0$ amd $b \prec \lc f$ is a \emph{top-\lcr} of $f$ (w.r.t. $g$).
\end{enumerate}
\end{definition}

First we can note that it is enough to consider \ltrs since \lcrs are taken care
of when generating new pairs:

\begin{lemma}
\label{lem:ltrs-are-enough}
Algorithm~\ref{alg:bba} terminates and computes a correct strong \stb for a
given set of geneators and a given momonmial order if we change the
corresponding normal form algorithms to consider only \ltrs.
\end{lemma}

\begin{proof}
We need to show that \lcrs are considered when adding new \gpts to the pair set
$P$. Assume $h$ is the outcome of an \ltr. If there exist possible \lcrs for $h$
w.r.t. $G$ then there is a $g \in G$ such that $\hm g \mid \hm h$, $\lc g
\nmid \lc h$ and $\lc h \nmid \lc g$. Thus $\gpoly h g$ is generated:
\[\gpoly h g = a h + b t g\]
where $a, b \in \Z$ such that $\gcd\left(\lc h, \lc g\right) = a \lc f + b \lc
g$ and $t = \frac{\hm h}{\hm g}$. Now we distinguish two cases:
\begin{enumerate}
\item\label{first-case} If $a=1$ then $\gpoly h g = h - b t g$ which is exactly the corresponding
\lcr of $h$ w.r.t. $g$.
\item If $a \neq 1$ then the \lcr of $h$ w.r.t. $g$, $h' = h - ctg$ for some
$c\in \Z, t\in\epr$, coresponds to the first step
of the Euclidean algorithm calculating $\gcd\left(\lc h, \lc g\right)$. If there
is no further reduction of $h'$ then Algorithm~\ref{alg:bba} generates a corresponding
\gpt between $h'$ and $g$. It follows that
\[h - ctg + \gpoly {h'} g = \gpoly h g.\]
If $h'$ is further reducible by some $g' \in G$ then we note the following:
First of all $g' \in G\setminus \{g\}$ since $\lc g \succ \lc{h'}$ thus there
cannot exist a \ltr or \lcr of $h'$ w.r.t. $g$. Now the reduction of $h'$ by
$g'$ is given as
\[h' - c't'g' = h - ctg - c't'g'\]
for some $c' \in \Z$ and $t' \in \epr$. Since $\hm h = \hm{tg} = \hm{t'g'}$ we
conclude that $ctg + c't'g'$ corresponds to a multiple of either $\spoly
{g}{g'}$ or $\gpoly {g}{g'}$ depending on divisibility of $\lc g$ and $\lc{g'}$.
Nonetheless, once we have a standard representation for $\spoly{g}{g'}$ or
$\gpoly{g}{g'}$ we can reset $h$ by $h - ctg - c't'g'$. Either the lead term or
the lead coefficient of $h$ increases in this process. Thus after finitely many
steps this process terminates and we reach either case~(\ref{first-case}) or
there is no further \lcr reduction possible.
\end{enumerate}
\end{proof}

This was the usual way \singular implemented standard basis reduction over the
integers. Clearly, this is due to historical reasons where the implementation
over $\Z$ was only thought of as a slight generalization of the computation
over fields.

\begin{example}
Recall Example~\ref{ex:local-order} from Section~\ref{sec:notation}: If we
allow only \ltrs when reducing $h = 4+x$ w.r.t. the strong \stb $G = \left\{2-x+y+x^2,
x-2y-x^2-xy-x^3\right\}$ then we get the following reduction table with $T_h =
G$ at the beginning.\begin{table}[h]
  \centering
  \def\arraystretch{1.2}
    \begin{tabular}{c|c|c}
    \toprule
    \multicolumn{1}{c|}{$h$} &
    \multicolumn{1}{c|}{$g$} &
    \multicolumn{1}{c}{$h$ \emph{added to} $T_h$?}\\
    \midrule
    $4+x$ & $2-x+y+x^2$ & \checkmark\\
    $3x-2y-2x^2$ & $x-2y-x^2-xy-x^3$ & \checkmark \\
    $4y+x^2+3xy+3x^3$ & $4+x$ & - \\
    $x^2+2xy+3x^3$ & $x-2y-x^2-xy-x^3$ & \checkmark \\
    $4xy+4x^3+x^2y+x^4$ & $4+x$ & - \\
    $4x^3+x^4$ & $4+x$ & -\\
    $0$ & - & -\\
    \bottomrule
    \end{tabular}
    \vspace*{5mm}
  \caption{Only \ltrs used in Mora normal form}
  \label{table:infinite-nf-only-lcrs}
\end{table}

If we compare it to Table~\ref{table:infinite-nf} we see one more reduction step
introduced by choosing $x-2y-x^2-xy-x^3$ as reducer in the second reduction
step, since \lcrs are not allowed.
\end{example}

Note that Lemma~\ref{lem:ltrs-are-enough} shows that \lcrs are, from the
theoretical point of view, not needed. A \lcr corresponds
to a first step in the Euclidean algorithm when calculating $\gcd\left(\lc f, \lc
g\right)$ which will be done in the algorithm when considering $\gpoly fg$.
Still, it has a strong impact on the performance of the algorithm in practice:
Cutting the lead coefficient down as much as possible means that the element
might be used more often for reduction purposes. Moreover, generating new \spts
and \gpts with it leads to lower lead coefficients there. In some sense \lcrs are
the bridge between an \ltr of $f$ by $g$ and the $\gpoly fg$. In one specific
situation we can go directly from one to another, without the need of a bridge. 
Lemma~\ref{lem:ltrs-are-enough} also gives a hint to this situation stated in
the following.
\footnote{This idea is also implemented in the computer algebra system
\macaulay. We have discovered it independently and since we have not
found any proof for the statement we give one here.}
during the reduction process over the integers.
\begin{lemma}
Let $f \in \epr$, $G \subset \epr$ finite and $g\in G$ such that
$\hm g = \hm f$. Then it holds that
\[\left\langle f,g\right\rangle
  = \left\langle \spoly fg, \gpoly fg\right\rangle.\]
  \label{lem:m2-replace-trick}
\end{lemma}

\begin{proof}
Let $u,v, d \in \Z$ such that
\begin{equation}
u\, \lc f + v\, \lc g = d = \gcd\left(\lc f, \lc g\right).
\label{eq:gcd}
\end{equation}
We can then write
\begin{center}
$
\begin{array}[]{rcccccc}
\gpoly fg &= & u& f& + &v& g.\\
\spoly fg &=& \frac{\lc g}{d}& f& - &\frac{\lc f}{d}&g.
\end{array}
$
\end{center}
In order to show the statement we have to proof that $(f,g)$ and $\left(\gpoly fg,
\spoly fg\right)$ generate the same $\Z$-lattice. So, in the above
representation of $\gpoly fg$ and $\spoly fg$ in terms of $f$ and $g$ we have to
show that the corresponding coefficient matrix is invertible, i.e. has
determinante $\pm 1 \in \Z$:
To see this we set
\[M := \begin{pmatrix} u & v\\ \frac{\lc g}{d} & -\frac{\lc
  f}{d}\end{pmatrix}.\]
Finally, we compute
\[\det\left(M\right) = - u\, \frac{\lc f}{d}
- v\, \frac{\lc g}{d} = -\frac 1d \left(u\, \lc f + v\, \lc g\right)
  \stackrel{(\ref{eq:gcd})}{=} -1.\]
\end{proof}

How to use Lemma~\ref{lem:m2-replace-trick} in \bba ? The idea is that whenever
we reduce a new element $f$ we check if there exists a reducer $g\in G$ with
$\hm f = \hm g$, but $\lc g \nmid \lc f$. In this situation we do a $2$-by-$2$
replacement:
\begin{enumerate}
\item Compute $\gpoly fg$ and replace $g\in G$ by $\gpoly fg \in G$. Clearly,
already genrated pairs with $g$ as generator have to adjusted respectively.
\item Compute $\spoly fg$ and replace $f$ by $\spoly fg$. Note that we have not
changed the degree of $f$, but probably only multiplied $f$ with some
coefficient. With the newly defined $f$ we again enter the reduction process and
see, if we can further reducer it.
\end{enumerate}
This has two main advantages to the usual reduction process that would compute
only an \lcr of $f$ w.r.t. $g$:
\begin{enumerate}
\item We directly compute the \gpt of $f$ and $g$
whereas the before mentioned \lcr would represent only one step in the Euclidean
algorithm for reaching $\gpoly fg$. So we can directly replace $g$ with
$\gpoly fg$ which leads to smaller coefficients and multiples during the pair
generation. Furthermore, $\gpoly fg$ reduces other elements at least in all
situations $g$ would reduce, but it can possibly fulfill more \ltrs due to its
smaller lead coefficient.

Furthermore, since $\hm{\gpoly fg} = \hm g$ we can replace all \spts already generated
with $g$, again giving smaller coefficients in upcoming reduction processes.
Even more, using $\gpoly fg$ we are possibly able to render more \spts useless
applying the chain criterion, due to the smaller lead coefficient.
\item For $f$ the advantage is that we are not stuck with a \lcr only, but we
can go on and perform the \ltr $\spoly fg$ and thus directly lower the lead term
without the need of adding $f$ to the basis, which would blowi up pair
generation.
\end{enumerate}

\begin{remark} \
\begin{enumerate}
\item Note that $\hm f = \hm g$ is an essential condition for the correctness of
Lemma~\ref{lem:m2-replace-trick}. If, for example, only $\hm g \mid \hm f$
holds such that $\lambda = \frac{\hm f}{\hm g} > 1 \in \epr$, then we can only
recover $f$ and $\lambda g$ via $\spoly fg$ and $\gpoly fg$, but we are no
longer able to recover $g$.
\item  Overall applying \lcrs has a huge effect on running time: In most examples we
    get a speedup factor of $3$. If we even apply coefficient reductions to the
    tail terms of the newly added element to the basis, we get another factor of
    $3-5$.
\end{enumerate}
\end{remark}

When we enter the reduction process of an element $f$ in \bba we search for
reducers in the following order:
\begin{enumerate}
\item Is there $g \in G$ for a \ltr of $f$? If so we can cut down the lead
term of $f$ without the need of multiplying $f$ with any coefficient $\neq 1$.
\item Is there are $g \in G$ fulfilling Lemma~\ref{lem:m2-replace-trick}? If so
we can cut down the lead term of some coefficient multiple of $f$ and we can
further replace $g$ by $\gpoly fg$ leading to a better reducer.
\item Is there $g \in G$ for a \lcr of $f$? We cannot cut down the lead term of
$f$, but at least we can reduce the lead coefficient before adding $f$ to $G$
and generating new \spts and \gpts.
\end{enumerate}

%% file: results.tex
\section{Computational results}
\label{sec:results}
In this section we present the new implementation for standard basis computation
over Euclidean domains in \singular $4-1-2$.\footnote{In the \singular
sources since git commit 7d2091affbf4b4a1a382e5eb0a47f66c0f3c42f7a.}
We compare it to the current implementations in the
computer algebra systems \macaulay (version $1.12$) and \magma (version $2.23$).
For the comparison we use benchmarks with different properties and behaviours.
All examples are computed with respect to the degree reverse lexicographical order.
All algorithms run single threaded, we use an Intel Core i7-6700 CPU with $3.4$ GHz
and $64$GB RAM. The machine runs Arch Linux with unmodified 4.18.12 kernel. For
the examples we refer to~\cite{singular-benchmarks}.

\begin{table}[h]
	\centering
  \def\arraystretch{1.2}
    \begin{tabular}{c||r|r|r|r}
    \toprule
    \multicolumn{1}{c||}{\textbf{Examples}} &
    \multicolumn{1}{c|}{\singular (Thm.~\ref{thm:buchberger-criterion})} &
    \multicolumn{1}{c|}{\singular (all pairs)} &
    \multicolumn{1}{c|}{\macaulay} &
    \multicolumn{1}{c}{\magma}\\
    \midrule
    Cyclic-6 & $0.330$ & $0.320$ & $4.708$ & $2.799$\\
    Cyclic-7 & $18,731.820$ & $5,636.210$ & out of memory & $366.060$\\[0.2em]
    Katsura-7 & $2.200$ & $2.250$ & $204.928$ & $251.630$\\
    Katsura-8 & $133.390$ & $135.360$ & $64,555.420$ & ($>24$h)\\
    Katsura-9 & $13,366.590$ & $12,951.160$ & ($>24$h) & ($>24$h)\\[0.2em]
    Eco-9 & $3.920$ & $4.050$ & $870.409$ & $22.520$\\
    Eco-10 & $38.760$ & $40.670$ & ($>24$h) & $289.540$\\[0.2em]
    F-633 & $0.140$ & $0.120$ &$14.982$ & $12.880$\\
    F-744 & $118.610$ & $117.890$ & ($>24$h) & ($>24$h) \\[0.2em]
    Noon-7 & $34.930$ & $32.700$ & ($>24$h) & ($>24$h)\\
    Noon-8 & $3,1390.060$ & $3,079.370$ & ($>24$h) & ($>24$h)\\[0.2em]
    Reimer-5 & $3.620$ & $3.590$ & out of memory & $1,932.400$\\
    Reimer-6 & $1,216.960$ & $1,232.530$ & out of memory & ($>24$h)\\[0.2em]
    Lichtblau & $1.910$ & $1.830$ & $69.536$ & $2,242.900$\\[0.2em]
    Bayes-148 & $9.970$ & $9.900$ & $117.635$ & $46.240$\\[0.2em]
    Mayr-42 & $212.320$ & $212.770$ & $218.635$ & $40.270$\\[0.2em]
    Yang-1 & $149.120$ & $147.250$ & $181.210$ & $50.330$\\[0.2em]
    Jason-210 & $47.010$ & $46.780$ & ($>24$h) & ($>24$h)\\
    \bottomrule
    \end{tabular}
	\caption{Benchmark timings given in seconds (``($>24$h)'' means that we have stopped the computation
      after at least $24$ hours.)}
	\label{table:syz-example}
\end{table}

We can see that \singular's new implementation is always faster than \macaulay.
Comparing it to \magma we see that \magma is way faster for \texttt{Cyclic-7}.
We assume that this is due to \magma using Faug\`ere's F4
algorithm~(\cite{fF41999}): Our implementation considers less \spts and \gpts, but
the reduction steps in higher degree are way slower than the linear algebra done
in \magma. We believe that there are more classes of examples where the linear
algebra attempt is more efficient than the polynomial arithmetic used in
\singular, still, we need to investigate this problem in more detail. We can see
that the F4 algorithm seems to be beneficial also for \texttt{Mayr-42} and \texttt{Yang-1}.
Nevetheless, for most of the other examples considered, \singular is by a
factor faster than \magma.

In the above table we list timings for \macaulay's Buchberger implementation
using polynomial arithmetic (as for \singular).
We also tested \macaulay's F4 implementation, but in most of the above examples
it was slower or just as fast as the polynomial arithmetic implementation. Due
to its much higher memory usage, we could not finish most of the examples on the
given machine. The only example where we have seen a better result was Mayr-42
where \macaulay's F4 algorithm finished in $185$ seconds, still using way more
memory than \magma and nearly $10$ times as much as \singular. In the given
example \singular's new algorithm uses, aside from \texttt{Cyclic-7}, the least memory
of all compared implementations.

For \singular we can see that generating all possible \spts and \gpts or only
the ones needed (recall Theorem~\ref{thm:buchberger-criterion}) does not have a
bigger influence on the computation for most of the examples.
Still, for the bigger examples like \texttt{Noon-8} or \texttt{Katsura-9} we see that applying
Theorem~\ref{thm:buchberger-criterion} leads to slightly slower computation. In
\texttt{Cyclic-7} we can even see a huge impact, the computation slows down by a factor
of more than $3$! It sems that having more pairs available at an earlier stage
of the algorithm is advantageous compared to having less pairs overall. In most
of the examples the algorithm taking care of all possible pairs does not even
compute more reductions, the product and chain criterion removes those
pairs which are really useless at some later stage. All in all it seems that
considering all possible \spts and \gpts leads to a more stable algorithm.

We found that the application of Lemma~\ref{lem:m2-replace-trick} becomes
sometimes a bottleneck. For example, always exchanging $f$ and $g$ by $\spoly
fg$ and $\gpoly fg$ has a huge drawback in the computation of a basis for
\texttt{Cyclic-7}, leading to a slow down of a factor of more than $3$. We found
that overall it is a good heuristic to apply Lemma~\ref{lem:m2-replace-trick} at
most $5$ times per a single reduction process. The \singular timings in the
above table correspond exactly this implementation.

%% file: conclusion.tex
\section{Conclusion}
We have presented new ideas for computing standard bases over Euclidean domains
without zero divisors.
The implementation of the corresponding algorithms is available in \singular. We
have seen that \singular is in general faster than \macaulay and \magma in
various examples.

Our next steps include an implementation of the new ideas in the open source C library
\gbl which implements Faug\`ere's F4 algorithm~(\cite{gbl}). Doing so we hope to
benefit from the new ideas and the fast linear algebra. Still, not all ideas
presented here are trivial to move to an F4-style algorithm.

Moreover, we still see a lot of zero reductions in higher degree which slow down
the computation. In order to tackle this problem, we work on a more general chain
criterion trying to exploit more of the structure of the input system. Even
more, a further attempt on signature-based computation over Euclidean rings,
see~(\cite{eppSigZ2017}), should be possible.

Finding better heuristics for the application of
Lemma~\ref{lem:m2-replace-trick} depending on the structure of the input systems
is also an interesting topic to study further. If applied in a ``good'' way it
can have a strong impact on the overall computation.

Another topic we are working on is to improve the implementation in \singular
for Euclidean domains with zero divisors. There, special care needs to be
taken of the annihilator polynomials.

%% file: gbz-arxiv.bbl
\begin{thebibliography}{10}

\bibitem{arnoldModular2003}
{Arnold, E. A.}
\newblock {Modular algorithms for computing Gr\"obner bases}.
\newblock {\em {Journal of Symbolic Computation}}, 35:403--419, April 2003.

\bibitem{bwkGroebnerBases1993}
{Becker, T.}, {Weispfenning, V.}, and {Kredel, H.}
\newblock {\em {Gr{\"o}bner Bases}}.
\newblock {Graduate Texts in Mathematics, Springer Verlag}, 1993.

\bibitem{bGroebner1965}
{Buchberger, B.}
\newblock {\em {Ein Algorithmus zum Auffinden der Basiselemente des
  Restklassenringes nach einem nulldimensionalen Polynomideal}}.
\newblock PhD thesis, University of Innsbruck, 1965.

\bibitem{bGroebnerCriterion1979}
{Buchberger, B.}
\newblock {A Criterion for Detecting Unnecessary Reductions in the Construction
  of Gr{\"o}bner Bases}.
\newblock In {\em {EUROSAM '79, An International Symposium on Symbolic and
  Algebraic Manipulation}}, volume~72 of {\em Lecture Notes in Computer
  Science}, pages 3--21. Springer, 1979.

\bibitem{buchberger2ndCriterion1985}
{Buchberger, B.}
\newblock {Gr{\"o}bner Bases: An Algorithmic Method in Polynomial Ideal
  Theory}.
\newblock pages 184--232, 1985.

\bibitem{bGroebner1965eng}
{Buchberger, B.}
\newblock {An Algorithm for Finding the Basis Elements of the Residue Class
  Ring of Zero Dimensional Polynomial Ideal (English translation of Bruno
  Buchberger's PhD thesis}.
\newblock {\em {Journal of Symbolic Computation}}, 41(3-4):475 -- 511, 2006.

\bibitem{gbl}
{Eder, C.}
\newblock {\em {gb}}, 2018.
\newblock {\url{https://github.com/ederc/gb}}.

\bibitem{singular-benchmarks}
{Eder, C.}
\newblock {\em {singular-benchmarks}}, 2018.
\newblock {\url{https://github.com/ederc/singular-benchmarks}}.

\bibitem{eppSigZ2017}
{Eder, C.}, {Pfister, G.}, and {Popescu, A.}
\newblock {On Signature-based Gr\"obner Bases over Euclidean Rings}.
\newblock In {\em {ISSAC 2017: Proceedings of the 2011 international symposium
  on Symbolic and algebraic computation}}, pages {141--148}, 2017.

\bibitem{fF41999}
{Faug\`ere, J.-C.}
\newblock {A new efficient algorithm for computing Gr\"obner bases (F4).}
\newblock {\em {Journal of Pure and Applied Algebra}}, {139}({1--3}):{61--88},
  {June} {1999}.
\newblock {\url{http://www-salsa.lip6.fr/~jcf/Papers/F99a.pdf}}.

\bibitem{gmInstallation1988}
{Gebauer, R.} and {M{\"o}ller, H. M.}
\newblock {On an installation of Buchberger's algorithm}.
\newblock {\em {Journal of Symbolic Computation}}, 6(2-3):275--286,
  October/December 1988.

\bibitem{grauert1972}
{Grauert, H.}
\newblock {\"Uber die Deformation isolierter Singularit\"aten analytischer
  Mengen}.
\newblock {\em {Inventiones Mathematicae}}, 15(3):171--198, 1972.

\bibitem{gpSingularBook2007}
{Greuel, G.-M.} and {Pfister, G.}
\newblock {\em {A {\sc Singular} Introduction to Commutative Algebra}}.
\newblock {Springer Verlag}, {2nd} edition, 2007.

\bibitem{hironaka11964}
{Hironaka, H.}
\newblock {Resolution of Singularities of an Algebraic Variety over a Field of
  Characteristic Zero: I}.
\newblock {\em {Annals of Mathematics}}, 79(1):109--203, 1964.

\bibitem{hironaka21964}
{Hironaka, H.}
\newblock {Resolution of Singularities of an Algebraic Variety over a Field of
  Characteristic Zero: II}.
\newblock {\em {Annals of Mathematics}}, 79(2):205--326, 1964.

\bibitem{kapur1988}
{Kandri-Rody, A.} and {Kapur, D.}
\newblock {Computing a Gröbner basis of a polynomial ideal over a Euclidean
  domain}.
\newblock {\em Journal of Symbolic Computation}, 6(1):37 -- 57, 1988.

\bibitem{lichtblau2012}
D.~Lichtblau.
\newblock {Effective computation of strong Gröbner bases over Euclidean
  domains}.
\newblock {\em {Illinois J. Math.}}, 56(1):177--194, 2012.

\bibitem{mora82}
{Mora, T.}
\newblock {An Algorithm to Compute the Equations of Tangent Cones}.
\newblock {EUROCAM 82, Lecture Notes in Comp. Sci.}, {1982}.

\bibitem{singular-parallel}
{Steenpa\ss , A.}
\newblock {\em {\texttt{parallel.lib}. {A} {\sc Singular} library for parallel
  processes}}, 2016.
\newblock {\url{http://www.singular.uni-kl.de}}.

\bibitem{singular-tasks}
{Steenpa\ss , A.}
\newblock {\em {\texttt{tasks.lib}. {A} {\sc Singular} library for task
  handling}}, 2016.
\newblock {\url{http://www.singular.uni-kl.de}}.

\bibitem{Wienand2011}
O.~Wienand.
\newblock {\em Algorithms for Symbolic Computation and their Applications -
  Standard Bases over Rings and Rank Tests in Statistics}.
\newblock PhD thesis, 2011.

\end{thebibliography}
